\documentclass [12pt]{amsart}
\usepackage{geometry}
\usepackage{mathrsfs}
\usepackage{xcolor}
\usepackage{graphicx}
\usepackage{amsthm}
\usepackage{amssymb}
\usepackage{enumerate}
\usepackage{tikz-cd}

\theoremstyle{plain}
\newtheorem{thm}{Theorem}[section]
\newtheorem{cor}[thm]{Corollary}
\newtheorem{lem}[thm]{Lemma}
\newtheorem{prop}[thm]{Proposition}
\newtheorem{defn}[thm]{Definition}
\newtheorem{exa}[thm]{Example}
\newtheorem{rem}[thm]{Remark}

\newtheorem{ques}[thm]{Question}

\begin{document}

\title [$EM-$Graded Rings]{$EM-$Graded Rings}

\author[{{T. Alraqad}}]{\textit{Tariq Alraqad}}

\address
{\textit{Tariq Alraqad, Department of Mathematics, University of Hail, Saudi Arabia.}}
\bigskip
{\email{\textit{t.alraqad@uoh.edu.sa}}}

 \author[{{H. Saber}}]{\textit{Hicham Saber}}

\address
{\textit{Hicham Saber, Department of Mathematics, University of Hail, Saudi Arabia.}}
\bigskip
{\email{\textit{hicham.saber7@gmail.com}}}

 \author[{{R. Abu-Dawwas }}]{\textit{Rashid Abu-Dawwas }}

\address
{\textit{Rashid Abu-Dawwas, Department of Mathematics, Yarmouk
University, Jordan.}}
\bigskip
{\email{\textit{rrashid@yu.edu.jo}}}


\date{}

\begin{abstract}
The main goal of this article is to introduce the concept of $EM-G-$graded rings. This concept is an extension of the notion of $EM-$rings. Let $G$ be a group and $R$ be a $G-$graded commutative ring.  The $G-$gradation of $R$ can be extended to $R[x]$ by taking the components $(R[x])_{\sigma}=R_{\sigma}[x]$. We define $R$ to be $EM-G-$graded ring if every homogeneous zero divisor polynomial has an annihilating content. We provide examples of $EM-G-$graded rings that are not $EM-$rings and we prove some interesting results regarding these rings.
\end{abstract}

 \maketitle

\section{Introduction}\label{int}
Throughout this article, all rings are commutative rings with one. The set of all zero divisors of a ring $R$ is $Z(R)$, the set of all regular elements is denoted by $reg(R)$, and the set of all idempotent elements is denoted by $E(R)$. A zero divisor polynomial $f(x)\in R[x]$ is said to have an annihilating content if there are $c_f\in Z(R)$ and $f_1\in reg(R[x])$ such that $f(x)=c_ff_1(x)$. Annihilating contents simplify computations of annihilator of polynomials. For instance, in \cite{ab17} annihilating contents are used to obtain some results related to the zero divisors graph of $R[x]$, where $R$ is a principal ideal rings. As a generalization, Abuosba and Ghanem \cite{ab19} introduce the concept of $EM-$rings as follows: A ring $R$ is called $EM-$ring if every zero divisor polynomial in $R[x]$ has an annihilating content. These rings were studied extensively in \cite{ab19} and \cite{gh18}.

Our goal in this article is to extend the notion of $EM-$rings by using the concept of graded rings. Let $G$ be a group. A ring $R$ with unity $1$, is said to be  $G$-graded if there exist additive subgroups $\{R_{\sigma}\mid \sigma\in G\}$ such that $\displaystyle R=\oplus_{\sigma\in G}R_{\sigma}$ and $R_{\sigma}R_{\tau}\subseteq R_{\sigma\tau}$ for all $\sigma,\tau\in G$. This gradation is denoted by $(R,G)$. The elements of $R_{\sigma}$ are called homogeneous of degree $\sigma$. The set of all homogeneous elements is denoted by $h(R)$, and the set of homogenous zero divisors is denoted by $hZ(R)$.  If $x\in R$, then $x$ can be written uniquely as $\displaystyle\sum_{\sigma\in G}x_{\sigma}$, where $x_{\sigma}$ is the component of $x$ in $R_{\sigma}$.

If $R$ is a $G-$graded ring, then $R[x]$ is $G-$graded with gradation defined by $(R[x])_\sigma=R_\sigma[x]$, $\sigma\in G$. Clearly, $R[x]=\oplus_{\sigma\in G}R_{\sigma}[x]$, and $R_{\sigma}[x]R_{\tau}[x]\subseteq R_{\sigma\tau}[x]$, for all $\sigma,\tau\in G$. Furthermore, $f(x)=\sum_{i=0}^{n}a_ix^i\in h(R[x])$ if and only if there exists $\sigma\in G$ such that $a_i\in R_{\sigma}$ for all $i$. i.e. all coefficients of $f$ are homogenous of the same degree. We will say that $R$ is an $EM-G-$graded ring if every zero divisor homogeneous polynomial in $R[x]$ has an annihilating content. It is evident that every $EM-$ring (that is $G-$graded) is an $EM-G-$graded ring, however the converse is not true as we will see in Examples \ref{e1} and \ref{e2}.

The next section includes some preliminary results on graded rings and $EM-$rings that will be needed in the following section. Section \ref{sec1} is devoted to the concept of $EM-G-$graded rings and its properties. We show that if $R_e$ is an $EM-$ring and the other components satisfy some certain condition then $R$ is $EM-G-$graded ring. We deduce that if $(R,G)$ is a cross product and $R_e$ is an $EM-$ring then $R$ is an $EM-G-$graded ring. Some extensions of $EM-G-$graded rings are also investigated. Some of these results are similar to results obtained about $EM-$ring in \cite{ab19}. We show that if $R$ is $EM-G$-graded ring then so is $R[x]$. We also prove that if $R$ is an $EM-G$-graded ring and  $S\subseteq h(R)$ is multiplicatively closed then $S^{-1}R$ is $EM-G-$graded ring. In this section we also obtain a nice result related to the idealization $R(+)R$ that generalizes \cite[Theorem 5.1]{gh18}. We show that $R$ is an $EM-$ring if and only if $R(+)R$ is an $EM-Z_2-$graded ring with gradation $H_0=R\oplus 0$ and $H_1=0\oplus R$.

\section{Preliminaries}\label{sec0}
In this section we list some results on graded rings and $EM-$rings that will be needed in the squeal. For a general reference on graded rings the reader is referred to \cite{Nastasescue}.

\begin{defn}\label{def0}
Let $G$ be a group. A ring $R$ with unity $1$, is said to be  $G$-graded if there exist additive subgroups $\{R_{\sigma}\mid \sigma\in G\}$ such that $\displaystyle R=\oplus_{\sigma\in G}R_{\sigma}$ and $R_{\sigma}R_{\tau}\subseteq R_{\sigma\tau}$ for all $\sigma,\tau\in G$.
\end{defn}
When $R$ is $G$-graded we denote that by $(R,G)$. The support of $(R,G)$ is defined as $supp(R,G)=\{\sigma\in G:R_{\sigma}\neq0\}$. If $x\in R$, then $x$ can be written uniquely as $\displaystyle\sum_{\sigma\in G}x_{\sigma}$, where $x_{\sigma}$ is the component of $x$ in $R_{\sigma}$. It is well known that $R_{e}$ is a subring of $R$ with $1\in R_{e}$. An ideal $A$ of a $G-$graded ring $R$ is called $G-$graded ideal provided that $A=\oplus_{\sigma\in G}(A\cap R_{\sigma})$. In case $A$ is a $G-$graded ideal of a $G-$graded ring $R$ then the factor ring $R/A$ is a $G-$graded ring with gradation defined by $(R/A)_{\sigma}=R_{\sigma}+A/A$.

For a polynomial $f(x)\in R[x]$, we denote by $C(f)$ the ideal generated by the coefficients of $f$. Abuosba and Ghanem \cite{ab19} noted that if $c_f$ is an annihilating content of $f$ then $Ann_{R[x]}(f(x))=Ann_{R[x]}(c_f)$ and $Ann_R(C(f))=Ann_R(c_f)$. Next we show that if $f$ is homogenous then $C(f)$ is a $G-$graded ideal.

\begin{lem}\label{l2} Let $R$ be a $G-$graded ring and consider the $G-$grading on $R[x]$ whose components are $R([x])_{\sigma}=R_{\sigma}[x]$, $\sigma\in G$. If $f(x)=\sum_{i=0}^na_ix^i\in h(R[x])$ then $C(f)$ is a $G-$graded ideal of $R$. \end{lem}

\begin{proof}
Since $f$ is homogeneous, there is $\tau\in G$ such that $\{a_0,\dots,a_n\}\subseteq R_{\tau}$. Let $w=\sum_{\sigma\in G}w_{\sigma}\in C(f).$ Then we can write $w$ as $w=\sum_{i=o}^nr_ia_i$. In addition, each $r_i$ can be written as $r_i=\sum_{\sigma\in G}r_{\sigma,i}$. So, we have
$$\sum_{\sigma\in G}w_{\sigma}=\sum_{i=0}^n\left(\sum_{\sigma\in G}r_{\sigma,i}\right)a_i=\sum_{\sigma\in G}\left(\sum_{i=0}^nr_{\sigma,i}a_i\right).$$
Now, for each $\sigma\in G$ we have $r_{\sigma}$ and $ia_i$ belong to $R_{\sigma}R_{\tau}\subseteq R_{\sigma\tau}$. Therefore $w_{\sigma}\in C(f)$, for all $\sigma\in G$. Hence $C(f)$ is $G-$graded ideal of $R$.
\end{proof}

\begin{lem}\label{l1} Let $R$ be a $G-$graded ring. If for every $\sigma\in supp(R,G)$, there exists $u_{\sigma}\in R_{\sigma}$ such that $R_{\sigma}=R_eu_{\sigma}$ and $Ann_{R_e}(u_{\sigma})=\{0\}$, then $reg(R_e[x])\subseteq reg(R[x])$.
\end{lem}
\begin{proof}
Let $h(x)=\sum_{i=0}^na_ix^i\in reg(R_e[x])$. Then $Ann_{R_e}(a_0,a_1,\dots,a_n)=\{0\}$. Let $t=\sum_{\sigma\in G}t_g\in Ann_R(a_0,a_1,\dots,a_n)$. Then for each $i$, we have $0=\sum_{\sigma\in G}t_{\sigma}a_i$. Since $t_{\sigma}a_i\in R_{\sigma}$, by the unique representation of $0$, we get $t_{\sigma}\in Ann_R(a_0,a_1,\dots,a_n)$ for each ${\sigma}\in G$. On the other hand, we have that for every ${\sigma}\in G$ there exists $s_{t_{\sigma}}\in R_e$, such that $t_{\sigma}=u_{\sigma}s_{t_{\sigma}}$.  Since $u_{\sigma}s_{t_{\sigma}}a_i=t_{\sigma}a_i=0$ and $Ann_{R_e}(a_0,\dots,a_n)=\{0\}$, we obtain that $s_{t_{\sigma}}a_i=0$ for each $i$. So, $s_{t_{\sigma}}\in Ann_{R_e}(a_0,\dots,a_n)=\{0\}$, for all ${\sigma}\in G$. This implies that $t=0$. Hence $h(x)\in reg(R[x])$.
\end{proof}

\begin{defn}\label{d22} A grading $(R,G)$ is called crossed product over $supp(R,G)$ if for every $\sigma\in supp(R,G)$, $R_\sigma$ contains a unit.\end{defn}

\begin{prop}\label{p2} \cite[Proposition 1.7]{Da10} If $(R,G)$ is crossed product over $supp(R,G)$, then for each $\sigma\in supp(R,G)$, $R_{\sigma}=R_eu$ for some unit $u\in R_{\sigma}$
\end{prop}

\begin{defn}\label{ae} Let $R$ be a ring with two gradations $(R,G)$ and $(R,G')$. Then $(R,G)$ is almost equivalent to $(R,G')$ if there exists an automorphism $\phi:R\longrightarrow R$ such that for each $\tau\in G'$ the exists $\sigma\in G$ such that $\phi(R_{\sigma})=R_{\tau}$.
\end{defn}

\begin{prop}\label{p3}\cite[Proposition 8.1.2]{Nastasescue} Let $R$ be a $G-$graded ring. If $S\subseteq h(R)$ is a multiplicatively closed subset of $R$, then the ring $S^{-1}R$ is $G-$graded by the gradation $(S^{-1}R)_{\lambda}=\{\frac{a}{s}\mid s\in S, a\in R\text{ such that } \lambda=(deg(s))^{-1}deg(a)\}$.
\end{prop}

The following lemma can be found in \cite[Lemma 2.3]{an85}
\begin{lem}\label{l3}
If $f\in R[x]$ with $C(f)=Ra$ then there exists $g\in R[x]$ with $C(g)=R$ and $f(x)=ag(x)$
\end{lem}

\begin{defn}\label{d2} Let $H$ be a subset (res. subring, ideal) of a ring $R$. Then $H$ is called $EM-$subset (res. $EM-$subring, $EM-$ideal) of $R$ if for each $f(x)\in H[x]\cap Z(R[x])$, there exists $c_f\in Z(R)$ and $g(x)\in reg(R[x])$ such that $f(x)=c_fg(x)$.
\end{defn}

\section{$EM-G-$graded rings}\label{sec1}

We start this section by introducing the concept of $EM-G-$graded ring.

\begin{defn}\label{d1} Let $R$ be a $G-$graded ring and consider the $G-$grading on $R[x]$ whose components are $R([x])_{\sigma}=R_{\sigma}[x]$, $\sigma\in G$. Then $R$ is called $EM-G-$graded ring if every non-zero polynomial in $hZ(R[x])$ has an annihilating content.
\end{defn}

We can see that every $EM-$ring that is $G-$graded is an $EM-G-$graded ring. The next example shows that the converse is not true.

\begin{exa}\label{e1} The ring $R=\mathbb{Z}_4[y]/(y^2)=\{a+bY: a,b\in \mathbb{Z}_4\text{ and } Y^2=0\}$ is not an $EM-$ring beacuse the polynomial $f(x)=2+Yx\in Z(R[x])$ does not have annihilating content (see \cite[Example 2.6]{ab19}). Now $R$ is $\mathbb{Z}_2-$graded by $R_0=\mathbb{Z}_4$, $R_1=\mathbb{Z}_4Y$. Let $f(x)=a_0+a_1x+\dots+a_nx^n\in hZ(R[x])$. Then either $\{a_0,\dots,a_n\}\subseteq \{0,2\}$ or  $\{a_0,\dots,a_n\}\subseteq \mathbb{Z}_4Y$. If $\{a_0,\dots,a_n\}\subseteq \{0,2\}$, then $a_n=2$, and hence $2$ is an annihilating content of $f$. Assume $\{a_0,\dots,a_n\}\subseteq \mathbb{Z}_4Y$. Then there is a polynomial $f_1(x)$ whose one of the coefficients is $1$ or $3$ such that $f(x)=Yf_1(x)$ or $f(x)=2Yf_1(x)$. Hence either $Y$ or $2Y$ is an annihilating content of $f$. Therefore $R$ is an $EM-\mathbb{Z}_2-$graded ring.
\end{exa}

\begin{thm}\label{t1} Let $R$ be $G-$graded ring. Then $R$ is an $EM-G-$graded ring if and only if $R_{\sigma}$ is an $EM-$subset of $R$ for all $\sigma\in G$.
\end{thm}
\begin{proof}Straightforward because $hZ(R[x])=\cup_{g\in G}(R_g[x]\cap Z(R[x]))$. \end{proof}

\begin{cor}\label{c1}
If $R$ is an $EM-G-$graded ring then $R_e$ is an $EM-$subring of $R$.
\end{cor}

\begin{thm}\label{t2} Let $R$ be a $G-$graded ring such that $R_e$ is an $EM$-ring. If for every $\sigma\in supp(R,G)$, there exists $u_\sigma\in R_\sigma$ such that $R_{\sigma}=R_eu_{\sigma}$ and $Ann_{R_e}(u_{\sigma})=\{0\}$, then $R$ is an $EM-G-$graded ring.
\end{thm}
\begin{proof}
Let $f(x)\in hZ(R[x])$. Then these exists $\sigma\in G$ such that $a_i\in R_{\sigma}$ for all $i$. Therefore $f(x)=u_{\sigma}h(x)$ for some $h(x)\in R_e[x]$. If $h(x)\in reg(R_e[x])$, then by Lemma \ref{l1}, $h(x)\in reg(R[x])$, and hence $u_{\sigma}$ is an annihilating content of $f$. Suppose $h(x)\in Z(R_e[x])$. Then there exist $c\in Z(R_e)$ and $h_1(x)\in reg(R_e[x])$ such that $h(x)=ch_1(x)$. So we have $f(x)=u_{\sigma}ch_1(x)$ where $u_{\sigma}c\in Z(R)$ and $h_1(x)\in reg(R[x])$. Hence $u_{\sigma}c$ is an annihilating content of $f(x)$.
\end{proof}

\begin{exa}\label{e2} The ring $R=\mathbb{Z}_6[x,y]/(xy)$ is not an $EM-$ring (see \cite[Example 3.11]{ab19}). Now $R$ is $\mathbb{Z}\times \mathbb{Z}-$graded by $R_{(0,0)}=\mathbb{Z}_6$, $R_{(i,0)}=\mathbb{Z}_6x^i$, $R_{(0,i)}=\mathbb{Z}_6y^i$ ($i\geq 1$), and $R_{(i,j)}=0$ otherwise. For each i,  $R_{(i,0)}=R_{(0,0)}x^i$, $R_{(0,i)}=R_{(0,0)}y^i$, and $Ann_{R_{(0,0)}}(x^i)=Ann_{R_{(0,0)}}(y^i)=\{0\}$. So by Theorem \ref{t2}, we get that $R$ is an $EM-\mathbb{Z}\times \mathbb{Z}-$graded ring.
\end{exa}

Let $R$ be any commutative ring (need not to be $G-$graded) and $n\geq2$. Then the ring $H=R[x]/(x^n)$ is $\mathbb{Z}_n-$graded by the gradation $H_k=Rx^k$, $k\in \mathbb{Z}_n$. For each $k\in \mathbb{Z}_n$, we have $H_k=Rx^k=H_0x^k$, and $Ann_R(x^k)=\{0\}$. So by Theorem \ref{t2}, we obtain following result

\begin{cor}\label{c6} Let $R$ be an $EM-$ring and $n\geq2$ be an integer. Then $H=R[x]/(x^n)$ is $EM-\mathbb{Z}_n-$graded by the gradation $H_k=Rx^k.$
\end{cor}

Let $R$ be a ring and $G$ be a group. Then the group ring $R[G]$ is $EM-G-$graded ring with gradation  $H_{\sigma}=R\sigma$, $\sigma\in G$. The following corollary also follows directly from Theorem \ref{t2}.

\begin{cor}\label{c66} Let $R$ be an $EM-$ring and $G$ be a group. Then $R[G]$ is $EM-G-$graded ring with gradation  $H_{\sigma}=R\sigma$, $\sigma\in G$.
\end{cor}

\begin{cor}\label{c2}
Let $R$ be a $G-$graded ring. If $(R,G)$ is crossed product over $supp(R,G)$ and $R_e$ is an $EM-$ring then $R$ is an $EM-G-$graded ring.
\end{cor}
\begin{proof} The result follows directly from Proposition \ref{p2} and Theorem \ref{t2}. \end{proof}

\begin{thm} Let $R$ be a ring with two gradations $(R,G)$ and $(R,G')$ such that $(R,G)$ is almost equivalent to $(R,G')$. If $R$ is $EM-G-$graded ring then $R$ is $EM-G'-$graded ring.
\end{thm}
\begin{proof}
Let $f(x)=\sum_{i=0}^{n}a_ix^i$ be a homogenous zero divisor of $R[x]$ under the gradation $(R,G')$. So $\sum_{i=0}^{n}\phi^{-1}(a_i)x^i$ a homogenous zero divisor of $R[x]$ under the gradation $(R,G)$. Hence there exists $c\in Z(R)$ and $g(x)\in reg(R[x])$ such that $\sum_{i=0}^{n}\phi^{-1}(a_i)x^i=cg(x)$. This implies that $\phi(c)$ is an annihilating content of $f$.
\end{proof} .

A ring $R$ is called Bezout ring if every finitely generated ideal is principal. $R$ is called Armendariz if the product of two polynomials on $R[x]$ is zero if and only if the product of their coefficients is zero. Abuosba and Ghanem \cite{ab19} showed that every Bezout ring is $EM-$ring and every $EM-$ring is Aremndariz. We extend these results into the notion of graded rings. We shall call a $G-$graded ring $R$, Bezout$-G-$graded if each finitely generated $G-$graded ideal is principal. In addition, we will say $R$ is  Armendariz$-G-$graded ring if the product of two polynomials in $hZ(R[x])$ is zero if and only if the product of their coefficients is zero.

\begin{cor}\label{c3} Every Bezout$-G-$graded ring is an $EM-G-$graded ring.
\end{cor}
\begin{proof} The result follows directly from Lemma \ref{l2} and Lemma \ref{l3}.
\end{proof}

\begin{thm}\label{t3} Every $EM-G-$graded ring is Aremndariz$-G-$graded.
\end{thm}
\begin{proof}Suppose $R$ is an $EM-G-$graded ring. Let $f(x)=\sum_{i=0}^na_ix^i$ and $g(x)=\sum_{i=0}^mb_ix^i$ be two polynomials in $h(R[x])$ such that $f(x)g(x)=0$. Since $R$ is an $EM-G-$graded, we get $f(x)=c_f\sum_{i=0}^{n_1}r_ix_i$ and $g(x)=c_g\sum_{j=0}^{m_1}s_jx^j$ where $c_fc_g=0$. Hence, for each $i,j$, we have $a_ib_j=c_fr_ic_gs_j=0$.
\end{proof}

\begin{thm}\label{t4} Let $R$ be an $EM-G-$graded ring and $S\subseteq h(R)$ be a multiplicatively closed subset of $R$. Then $S^{-1}R$ is $EM-G-$graded ring by the gradation $(S^{-1}R)_{\lambda}=\{\frac{a}{s}\mid s\in S, a\in R\text{ such that } \lambda=(deg(s))^{-1}deg(a)\}$.
\end{thm}
\begin{proof} Assume $f(x)\in hZ(S^{-1}R[x])$. Then $f(x)=\frac{h(x)}{t}$ for some $t\in S$ and some $h(x)\in hZ(R[x])$. Since $R$ is an $EM-G-$graded ring, there exist $c\in Z(R)$ and $g(x)=\sum_{i=0}^nb_ix^i\in reg(R[x])$ such that $h(x)=cg(x)$. So $f(x)=c\frac{g(x)}{t}$. Suppose that there exists $\frac{d}{k}\in S^{-1}R$ such that $\frac{d}{k}\frac{g(x)}{t}=0$. Then there exists $u\in S$ such that $udg(x)=0$. Since $g(x)$ is regular in $R[x]$ and $u\in S$, we get $d=0$. Hence $\frac{g(x)}{t}$ is regular in $S^{-1}R[x]$. Therefore $S^{-1}R[x]$ is an $EM-G-$graded ring.
\end{proof}

If $R$ is a commutative ring with unity, then $reg(R)$ is muliplicatively closed subset of $R$, and the localization $T(R)=(reg(R))^{-1}R$ is called the total quotient ring of $R$. Abuosba and Ghanem \cite{ab19} deduced that if $R$ is an $EM-$ring then so is $T(R)$. A similar result is obtain for $EM-G-$graded rings. Let $hreg(R)$ be the set of regular homogenous elements of $R$ and define $hT(R)=(hreg(R))^{-1}R$.

\begin{cor}\label{c4} If $R$ is $EM-G-$graded ring then so is $hT(R)$.\end{cor}

\begin{thm}\label{t5} Let $R$ be a $G-$graded ring. If $hT(R)$ is $EM-G-$graded ring then for every $f(x)\in hZ(R[x])$ there exists $c\in R$ such that $Ann_{R[x]}(f)=Ann_{R[x]}(c)$.
\end{thm}

\begin{proof} Let $f(x)\in hZ(R[x])$. Then $\frac{f(x)}{1}\in hZ(hT(R)[x])$. So there exist $\frac{c}{t}\in Z(hT(R)[x])$, and $g(x)\in reg(hT(R)[x])$ such that $\frac{f(x)}{1}=\frac{k}{t}g(x)$. So we have $u(tf(x)-kg(x))=0$ for some  $u\in hreg(R[x]).$ Since $u$, $t$, and $g(x)$ are regular, we get $Ann_{R[x]}(f)=Ann_{R[x]}(k)$.
\end{proof}

Let $R_{\alpha}$, $\alpha\in I$ be a family of $G-$graded rings. Then $R=\prod_{\alpha\in I}R_{\alpha}$ is $G-$graded by the gradation $R_g=\prod_{\alpha\in I}(R_{\alpha})_g$, $g\in G$. Next we generalize \cite[Theorem 3.12]{ab19}.

\begin{thm}\label{t6}Let $R_{\alpha}$, $\alpha\in I$ be a family of $G-$graded rings, and $R=\prod_{\alpha\in I}R_{\alpha}$. Then $R$ is $EM-G-$graded if and only if $R_{\alpha}$ is $EM-G-$graded for each $\alpha\in I$.
\end{thm}

\begin{proof} Suppose $R$ is $EM-G-$graded. Fix $j\in I$, and let $f(x)=\sum_{i=0}^na_ix^i\in hZ(R_j[x])$. Consider the polynomial $g(x)=\sum_{i=0}^nb_{\alpha,i}x^i\in R[x]$, where for each $i$, $b_{j,i}=a_i$, and $b_{\alpha,i}=0$, for all $\alpha\neq j$. Since all $a_i's$ are in the same component of the $G-$grading of $R_{\alpha}$, we get $g(x)\in hZ(R[x])$. Hence there exist $c=(c_{\alpha})\in Z(R)$ and $g_1(x)=\sum_{i=0}^m (d_{\alpha,i})x^i \in reg(R[x])$ such that $g(x)=cg_1(x)$, and $m\geq n$. Then $f(x)=c_j\sum_{i=0}^md_{j,i}x^i$. It is clear that $c_j\neq0,$ because $0\neq a_n=c_jd_{j,n}$. Assume $yd_{j,i}=0$ for all $i$. Let $q=(q_{\alpha})$ where $q_j=y$ and $q_{\alpha}=0,$ when $\alpha\neq j$. Then $(q_{\alpha})(d_{\alpha,i})=0$ for all $i$. So $y=0$ and hence $\sum_{i=0}^n d_{j,i}x^i \in reg(R_{\alpha}[x])$.

For the converse assume $R_{\alpha}$ is $EM-G-$graded for all $\alpha\in I$. Let  $f(x)=\sum_{i=0}^n(a_{\alpha,i})x^i\in hZ(R[x])$. Then there exists $(b_{\alpha})\in R[x]$ such that $(b_{\alpha})(a_{\alpha,i})=(0)$ for all $i$. For each $\alpha\in I$, let $f_{\alpha}(x)=\sum_{i=0}^na_{\alpha,i}x^i$ and let $J=\{\alpha\in I\mid f_{\alpha}\notin reg(R_{\alpha})\}$. So, for each $\alpha\in J$, we have $f_{\alpha}(x)\in hZ(R_j[x])$, and hence $f_{\alpha}=c_{\alpha}g_{\alpha}(x)$ for some $c_{\alpha}\in Z(R_{\alpha})$ and some $g_{\alpha}(x)\in reg(R_{\alpha}[x])$. Now for $\alpha\notin J$, let $c_{\alpha}=1$ and $g_{\alpha}(x)=f_{\alpha}(x)$. Hence we have $f(x)=(c_{\alpha})(g_{\alpha}(x))$, where $(c_{\alpha})\in Z(R)$ and $(g_{\alpha}(x))\in reg(R[x])$ as desired.
\end{proof}

\begin{thm}\label{t7} If $R$ is an $EM-G-$graded ring then so is $R[x]$.
\end{thm}
\begin{proof} Let $f(x)=\sum_{i=0}^nf_i(x)y^i\in hZ(R[x,y])$. Then there exists $h(x)\in R[x]$ such that $h(x)f_i(x)=0$ for all $i$. Let
\[g(x)=f_0(x)+f_1(x)x^{deg(f_0)+1}+f_2(x)x^{deg(f_0)+deg(f_1)+2}+\dots+f_n(x)x^{\sum_{i=0}^{n-1}deg(f_i)+n}.\]
So $hg=0$. Moreover, since $f(x,y)\in h(R[x,y])$, there exists $\lambda\in G$ such that $R_{\lambda}$ contains all coefficients of $f_i(x)$, for all $i$. Hence $g(x)\in hZ(R[x])$. Thus there exists $c\in Z(R)$ and $g_1(x)=\sum_{i=0}^mb_ix^i\in reg(R[x])$ such that $g(x)=cg_1(x)$. So $f_0(x)=c\sum_{i=0}^{deg(f_0)}b_ix^i=cw_0(x)$, $f_1(x)=c\sum_{i=0}^{deg(f_1)}b_{i+deg(f_0)+1}x^i=cw_1(x)$, and so on. Thus $f(x,y)=c\sum_{i=0}^{n}w_i(x)y^i$, and since $\cap Ann(b_i)=0$, we have $\sum_{i=0}^{n}w_i(x)y^i\in reg(R[x,y])$. Therefore $R[x]$ is $EM-G-$graded.
\end{proof}

\begin{cor}\label{c5} Let $R$ be an $EM-G-$graded ring and $k$ be a positive integer. Then $R[x_1,x_2,\dots,x_k]$ is an $EM-G-$graded ring.
\end{cor}

\begin{rem}\label{r1} If $R$ is a $G-$graded ring then  $(x^2)$ is a $G-$graded ideal of $R[x]$. So $R[x]/(x^2)$ is a $G-$graded ring with gradation $(R[x]/(x^2))_g=R_g[x]+(x^2)/(x^2)$ i.e. $a+bX\in h(R[x]/(x^2))$ if and only if there exists $\sigma\in G$ such that $a,b\in R_\sigma$.
\end{rem}

\begin{thm}\label{t8} Let $R$ be $G-$graded such that $h(R)\cap Z(R)=\{0\}$. Then $H=R[x]/(x^2)$ is an $EM-G-$graded ring.
\end{thm}
\begin{proof}
Let $f(y)=\sum_{i=0}^n(a_i+b_iX)y^i\in hZ(H[y])$. Then by the argument in Remark \ref{r1} we get that $a_i,b_i\in h(R)$ and $a_i\in Z(R)$ for all $i$. Hence $a_i=0$ for all $i$. So we have $f(y)=X\sum_{i=0}^nb_iy^i$, which yields $X$ is an annihilating content for $f$. Thus $H$ is an $EM-G-$graded ring.
\end{proof}

\begin{thm}\label{t9} Let $R$ be a $G-$graded ring. If $H=R[x]/(x^2)$ is an $EM-G-$graded ring, then so is $R$.
\end{thm}
\begin{proof}
Let $f(y)=\sum_{i=0}^nr_iy^i\in hZ(R[y])$. Then $f(y)\in hZ(H[y])$. Therefore $f(y)=(c+dX)\sum_{i=0}^{m}(a_i+b_iX)y^i$, with $\cap_{i=0}^mAnn(a_i+b_iX)=\{0\}$. So we have $\cap_{i=0}^mAnn(a_i)=\{0\}$ and $r_i=ca_i$ for each $i$. Hence $f(y)=c\sum_{i=0}^na_iy^i$. Therefore $R$ is an $EM-G-$graded ring.
\end{proof}

\begin{thm}\label{t10} Let $R$ be $G-$graded ring such that for each $a\in h(R)$ there exists $b\in E(R)$ such that $ann(a)=bR$. Then $R$ is an $EM-G-$graded ring.
\end{thm}
\begin{proof} Let $f(x)=\sum_{i=0}^na_ix^i\in hZ(R[x])$. Then $a_i\in h(R)$ for all $i$. So for each $i$ there exists $b_i\in E(R)$ and $u_i\in reg(R)$ such that $a_i=u_ib_i$. Let $b=1-\prod_{i=0}^n(1-b_i)$. Clearly $1-b\in E(R)$, and so $b\in E(R)$. Since $b_i(1-b_i)=0$, we get $b(b_i+1-b)=bb_i=b_i-b_i\prod_{j=0}^n(1-b_j)=b_i$ and $1=\sum_{i=0}^n(b_i+1-b)-\prod_{i=0}^n(b_i+1-b)$. Hence $Ann(b_0+1-b,b_1+1-b,\dots,b_n+1-b)=\{0\}$ and for each $i$, $a_i=bu_i(b_i+1-b).$  Thus $f(x)=bg(x)$ where $b\in Z(R)$ and $g(x)=\sum_{i=0}^nu_i(b_i+1-b)x^i\in reg(R[x])$. Therefore $R$ is $EM-G-$graded ring.
\end{proof}

Let $R$ be a ring and $M$ be an $R-$module. Then the idealization $R(+)M$ is the ring whose elements are those of $R\times M$ equipped with addition and multiplication defined by $(r_1,m_1)+(r_2,m_2)=(r_1+r_2,m_1+m_2)$ and $(r_1,m_1)(r_2,m_2)=(r_1r_2,r_1m_2+r_2m_1)$ respectively. The annihilator of the module $M$ in $R$ is the set $Ann_R(M)=\{r\in R\mid rM=\{0\}\}$. It is well known that the idealization $R(+)M$ is $\mathbb{Z}_2-$graded by the gradation $(R(+)M)_0=R\oplus 0$ and $(R(+)M)_1=0\oplus M$.

\begin{thm}\label{t11} Let $R$ be a ring and $M$ be an $R-$module such that $Ann_R(M)=\{0\}$. If $R(+)M$ is an $EM-\mathbb{Z}_2-$graded by the gradation $H_0=R\oplus 0$ and $H_1=0\oplus M$, then $R$ is an $EM-$ring.
\end{thm}
\begin{proof} Suppose $R(+)M$ is an $EM-\mathbb{Z}_2-$graded by the gradation $H_0=R\oplus 0$ and $H_1=0\oplus M$. Let $f(x)=\sum_{i=0}^na_ix^i\in Z(R[x])$. Then $Ann_{R(+)M}\{(a_i,0)\}\neq\{0\}$. So $\sum_{i=0}^n(a_i,0)=x^i=(c,m)\sum_{i=0}^k(b_i,m_i)x^i$ with $k\geq n$ and $Ann_{R(+)M}\{(b_i,m_i)\}=\{0\}$. Suppose $0\neq w\in Ann_R\{\alpha_i\}$. Since $Ann_R(M)=\{0\}$, there exists $t\in M$ such that $wt\neq0$. However, since $(0,wt)(b_i,m_i)=(0,wtb_i)=(0,0)$ for all $i$, we get $(0,wt)\in Ann_{R(+)M}\{(b_i,m_i)\}=\{0\}$, a contradiction. Therefore $f(x)=c\sum_{i=0}^nb_ix^i$ with $Ann_R\{b_i\}=\{0\}$ as desired.
\end{proof}

Ganam and Abuosba in \cite{gh18} proved that if $R(+)R$ (equivalently $R[x]/(x^2)$) is an $EM-$ring then so is $R$. However the converse of this  result is not true. From Theorem \ref{t2} and Theorem \ref{t11} we obtain the following result.

\begin{cor}\label{c7} A ring $R$ is an $EM-$ring if and only if $H=R(+)R$ is an $EM-Z_2-$graded ring with gradation $H_0=R\oplus 0$ and $H_1=0\oplus R$.
\end{cor}

\section{Further Questions}\label{sec3}

In this section we highlight two problems that may be of interest for future research.

Several results on $EM-$rings can be extended to $EM-G-$graded rings if we can show that every homogenous zero divisor polynomial has at least one homogenous annihilating content. The rings discussed in Theorem \ref{t2}, Corollary \ref{c6}, and Corollary \ref{c66} have this property, however we don't know if this always the case. Based on this observation we ask the following question.

\begin{ques} In an $EM-G-$graded ring $R$, is it guaranteed that every polynomial in $hZ(R[x])$ has an annihilating content that belongs to $hZ(R)$? If not, then what kinds of graded rings have this property?
\end{ques}

The second problem is related to the notion of strongly $EM$ rings that was defined in \cite{ab19}. An $EM$ ring $R$ is called strongly $EM$ ring if every zero divisor power series has an annihilating content. If $R$ is $G-$graded ring then $R[[x]]$ is $G-$graded by the gradation $(R[[x]])_g=R_g[[x]]$, $g\in G$. We define an $EM-G-$graded ring $R$ to be strongly $EM-G-$graded ring if every nonzero homogeneous power series in $R[[x]]$ has an annihilating content. It is not guaranteed that if every power series in $hZ(R[[x]])$ has an annihilating content then every polynomial in $hZ(R[x])$ has an annihilating content. So similar to a question from \cite{ab19}, we ask the following question.

\begin{ques} What types of graded rings have the property that if every power series in $hZ(R[[x]])$ has an annihilating content then every polynomial in $hZ(R[x])$ has an annihilating content?
\end{ques}

The following theorem describes one of these types.

\begin{thm}\label{t12}  Let $R$ be a $G-$graded ring such that $Z(R[[x]])$ is an ideal of $R[[x]]$. Then $R$ is strongly $EM-G-$ring if and only if every homogeneous power series in $R[[x]]$ has an annihilating content.
\end{thm}
 \begin{proof} Suppose that every homogeneous power series in $R[[x]]$ has an annihilating content. Let $f(x)=\sum_{i=0}^{n}a_ix^i\in hZ(R[[x]])$. Then $f(x)=c_ff_1$ for some $c_f\in Z(R)$ and $f_1(x)=\sum_{i=0}^{\infty}b_ix^i\in reg(R[[x]])$. Thus we have, $a_i=c_fb_i$, for $i=0,1,\dots, n$, and $0=c_fb_i$, for $i\geq n+1$. So $\sum_{i=n+1}^{\infty}b_ix^i\in Z(R[[x]])$. Since $Z(R[[x]])$ is an ideal of $R[[x]]$ and $f_1(x)\in reg(R[[x]])$, we have $\sum_{i=0}^{n}b_ix^i\in reg(R[[x]])$. Thus $f(x)=c_f\sum_{i=0}^{n}b_ix^i$, where $\sum_{i=0}^{n}b_ix^i\in reg(R[x])$. Therefore $R$ is $EM-G-$graded.
\end{proof}

\end{document}